\documentclass[11pt]{amsart}
\usepackage[english]{babel}
\usepackage[utf8]{inputenc}
\usepackage{graphicx}
\usepackage{float,rotating,import}
\usepackage{amsmath, amssymb, amsthm,color,todonotes}
\usepackage{indentfirst}
\usepackage[a4paper, top=3cm, bottom=3cm, left=2.5cm, right=2.5cm]{geometry}
\usepackage{tikz}

\newcommand{\rOnePlus}{\raisebox{-5pt}{
\begingroup%
  \makeatletter%
  \providecommand\color[2][]{%
    \errmessage{(Inkscape) Color is used for the text in Inkscape, but the package 'color.sty' is not loaded}%
    \renewcommand\color[2][]{}%
  }%
  \providecommand\transparent[1]{%
    \errmessage{(Inkscape) Transparency is used (non-zero) for the text in Inkscape, but the package 'transparent.sty' is not loaded}%
    \renewcommand\transparent[1]{}%
  }%
  \providecommand\rotatebox[2]{#2}%
  \ifx\svgwidth\undefined%
    \setlength{\unitlength}{18.70443773bp}%
    \ifx\svgscale\undefined%
      \relax%
    \else%
      \setlength{\unitlength}{\unitlength * \real{\svgscale}}%
    \fi%
  \else%
    \setlength{\unitlength}{\svgwidth}%
  \fi%
  \global\let\svgwidth\undefined%
  \global\let\svgscale\undefined%
  \makeatother%
  \begin{picture}(1,0.91836368)%
    \put(0,0){\includegraphics[width=\unitlength,page=1]{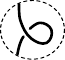}}%
  \end{picture}%
\endgroup%
}}
\newcommand{\rOneMinus}{\raisebox{-5pt}{
\begingroup%
  \makeatletter%
  \providecommand\color[2][]{%
    \errmessage{(Inkscape) Color is used for the text in Inkscape, but the package 'color.sty' is not loaded}%
    \renewcommand\color[2][]{}%
  }%
  \providecommand\transparent[1]{%
    \errmessage{(Inkscape) Transparency is used (non-zero) for the text in Inkscape, but the package 'transparent.sty' is not loaded}%
    \renewcommand\transparent[1]{}%
  }%
  \providecommand\rotatebox[2]{#2}%
  \ifx\svgwidth\undefined%
    \setlength{\unitlength}{18.70443773bp}%
    \ifx\svgscale\undefined%
      \relax%
    \else%
      \setlength{\unitlength}{\unitlength * \real{\svgscale}}%
    \fi%
  \else%
    \setlength{\unitlength}{\svgwidth}%
  \fi%
  \global\let\svgwidth\undefined%
  \global\let\svgscale\undefined%
  \makeatother%
  \begin{picture}(1,0.91836368)%
    \put(0,0){\includegraphics[width=\unitlength,page=1]{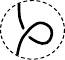}}%
  \end{picture}%
\endgroup%
}}
\newcommand{\rOneZero}{\raisebox{-5pt}{
\begingroup%
  \makeatletter%
  \providecommand\color[2][]{%
    \errmessage{(Inkscape) Color is used for the text in Inkscape, but the package 'color.sty' is not loaded}%
    \renewcommand\color[2][]{}%
  }%
  \providecommand\transparent[1]{%
    \errmessage{(Inkscape) Transparency is used (non-zero) for the text in Inkscape, but the package 'transparent.sty' is not loaded}%
    \renewcommand\transparent[1]{}%
  }%
  \providecommand\rotatebox[2]{#2}%
  \ifx\svgwidth\undefined%
    \setlength{\unitlength}{18.70443773bp}%
    \ifx\svgscale\undefined%
      \relax%
    \else%
      \setlength{\unitlength}{\unitlength * \real{\svgscale}}%
    \fi%
  \else%
    \setlength{\unitlength}{\svgwidth}%
  \fi%
  \global\let\svgwidth\undefined%
  \global\let\svgscale\undefined%
  \makeatother%
  \begin{picture}(1,0.91836368)%
    \put(0,0){\includegraphics[width=\unitlength,page=1]{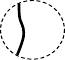}}%
  \end{picture}%
\endgroup%
}}

\newtheorem{Teo}{Theorem}[section]
\newtheorem{Def}[Teo]{Definition}
\newtheorem{Prop}[Teo]{Proposition}
\newtheorem{Lema}[Teo]{Lemma}
\newtheorem{Cor}[Teo]{Corollary}
\newtheorem{Ex}[Teo]{Example}

\newtheorem{Rem}[Teo]{Remark}

\begin{document}

\title[Geometric realization of almost-extreme Khovanov of semiadequate links]{\textsc{Geometric realization of the almost-extreme Khovanov homology of semiadequate links}}

\author{J\'{o}zef H. Przytycki}
\address{Department of Mathematics at The George Washington University, Washington DC, USA $\&$ University of Gda\'{n}sk, Poland.}
\email{przytyck@gwu.edu}
	
\author{Marithania Silvero}
\address{Institute of Mathematics of the Polish Academy of Sciences, Warsaw, Poland $\&$ Barcelona Graduate School of Mathematics at Universitat de Barcelona, Spain.}
\email{marithania@us.es}

\maketitle


\noindent \textbf{Abstract} \,
We introduce the notion of partial presimplicial set and construct its geometric realization. We show that any semiadequate diagram yields a partial presimplicial set leading to a geometric realization of the almost-extreme Khovanov homology of the diagram. We give a concrete formula for the homotopy type of this geometric realization, involving wedge of spheres and a suspension of the projective plane.

\bigskip

\noindent \textbf{Keywords:} Khovanov homology, semiadequate knot, geometric realization, homotopy type, presimplicial set. 
\mbox{ }


\section{Introduction}
Khovanov homology is a powerful link invariant introduced by Mikhail Khovanov as a categorification of the Jones polynomial \cite{Kho}.  The geometric realizations of Khovanov homology were first proposed by Chmutov (2005) and Everitt and Turner (2011), see \cite{Chm} and \cite{ET}, respectively.

The ultimate geometric realization was given by Lipschitz and Sarkar in \cite{LS}. Given a link diagram $D$, they constructed a Khovanov suspension spectrum, $\mathcal{X}_{Kh}(D)$, such that its reduced cellullar cochain complex is isomorphic to the Khovanov complex of $D$. Moreover, they proved that the stable homotopy type of $\mathcal{X}_{Kh}(D)$ is an invariant of the isotopy class of the corresponding link.

In the case of extreme Khovanov homology, the concrete geometric realization of Khovanov homology was introduced and developed in \cite{GMS} and further in \cite{PSil}, where it was conjectured (and proved for some special cases) that it has the homotopy type of a wedge of spheres.

In this paper we give a simple construction of geometric realization of almost-extreme Khovanov complex of semiadequate links and describe the precise homotopy type of this space. For this purpose we introduce the notion of partial presimplicial sets, a generalization of presimplicial sets. This allows us to exactly determine the homotopy type of our spaces. 

The main result of this paper is the following.

\begin{Teo}\label{teomain}
Let $D$ be an $A$-adequate link diagram and $G=G_A(D)$ the state graph associated to the state $s_A$. Then, the almost-extreme Khovanov complex of $D$ is chain homotopy equivalent to the cellular chain complex of 
\begin{itemize}
\item[(i)] $\bigvee_{p_1}S^{c-2} \vee S^{c-1}$  if $G$ is bipartite, 
\item[(ii)]$\bigvee_{p_1 - 1} S^{c-2} \vee \sum^{c-3} \mathbb{R}P^2$ otherwise,
\end{itemize} 
where $c$ is the number of crossings of $D$ and $p_1$ the cyclomatic number of the graph obtained from $G$ by replacing each multiedge by a single edge. 
\end{Teo}

Throughout the paper we work in PL category. Moreover, all homology of topological spaces are reduced.

The plan of the paper is as follows. In Section 2 we introduce the notion of partial presimplicial set and its geometric realization. We briefly review Khovanov homology and define almost-extreme Khovanov complex in Section 3. Section 4 is devoted to almost-extreme Khovanov homology of semiadequate links. In Section 5 starting from a semiadequate diagram $D$, we give a method to construct an associated partial presimplicial set $\mathcal{X}_D$ and study its homotopy type. Finally, in Section 6 we show that the cellular chain complex of the geometric realization of $\mathcal{X}_D$ is equivalent to the almost-extreme Khovanov complex of $D$. 

\section{Partial presimplicial sets}

In this section we introduce the notion of partial presimplicial set and its geometric realization. Partial presimplicial sets are special cases of presimplicial modules, but unlike general presimplicial modules they can be easily endowed with a geometric realization that will be the key point of the construction of the geometric realization of almost-extreme Khovanov homology introduced in Section \ref{FromDtoXD}. 

Partial presimplicial modules and simplicial modules were introduced by Eilenberg and Zilber in \cite{EZ} under the name of semisimplicial complexes and complete semisimplicial complexes, respectively (compare \cite{Lod}).

The definition of \textit{partial presimplicial sets} is based on two equivalent conditions: the first one uses properties of partial functions and the second ones places partial presimplicial sets in the context of presimplicial modules.

\begin{Def} A partial presimplicial set $\mathcal X = (X_n,d_i)$ consists of a sequence of sets $\{X_{n}\}_{n \geq 0}$ together with a collection of partial maps defined only on subsets $\operatorname{Dom}(d_{i,n})$ of $X_n$, $d_{i,n} = d_i: \operatorname{Dom}(d_{i,n}) \rightarrow X_{n-1}$, satisfying the following two equivalent conditions: 
\begin{itemize}
\item[(1)] For $0\leq i<j \leq n$,  either $d_id_j$ and $d_{j-1} d_i$ are defined and satisfy $d_id_j = d_{j-1} d_i$, or they are both undefined.   
\item[(2)] The natural extension\footnote{Any partial presimplicial set $\mathcal{X} = (X_n,d_i)$ can be naturally endowed with the structure of a presimplicial module by setting $C_n=RX_n$, that is, the free $R$-module with basis $X_n$, and extending $d_{i,n}$ to the homomorphism $C_n \rightarrow C_{n-1}$ by declaring $d_{i,n}(a)=0$ for $a \in X_n \setminus Dom(d_{i,n})$. In this way, a presimplicial module leads to the chain complex $(C_n,\partial_n)$ by defining $\partial_n=\sum_{i=0}^n(-1)^id_i$ and further to a homology (and cohomology) module.} of $(X_n, d_i)$, $(RX_n,d_i)$, is a presimplicial module and for any $x_n \in X_n$ either $d_i(x_n) \in  X_{n-1}$ or $d_i(x_n)=0$.
\end{itemize}
\end{Def}

We say that a partial presimplicial set is \textit{proper} if it is not a presimplicial set.

The standard geometric realization of presimplicial sets, where $X_n$ are names of simplexes which are glued by following instructions given by the face maps $d_i$, can be extended to the case of partial presimplicial sets in the following way\footnote{The geometric realization of partial presimplicial sets can be defined in various different ways; we describe here the \textit{pointed geometric realization}, which we find suitable for the goal of this paper.}:

\begin{Def}Let $\Delta^n=\{(x_0,...,x_n)\in R^{n+1} \ | \ \Sigma_{i=0}^n x_i=1,\ x_i\geq 0 \}$ be the ``model simplex'' and $d^{i,{n-1}}=d^i: \Delta^{n-1} \to \Delta^n$ the embedding maps given by $d^i(x_0,...,x_{n-1})= (x_0,...,x_{i-1},0,x_i,...,x_{n-1})$, for $0 \leq i \leq n$.  The (pointed) geometric realization of the proper partial presimplicial set $\mathcal{X} = (X_n, d_i)$, $|\mathcal{X}|$, is the $CW$-complex defined as $$|{\mathcal X}| = b \sqcup \bigsqcup_{n\geq 0} X_n \times \Delta^n \diagup \sim,$$ where the relation $\sim$ is given by $(d_i(a), t) \sim (a,d^i(t))$ for $a \in X_n$, $t \in \Delta^{n-1}$ if $d_i(a) \neq 0$, and $(a,d^i(t)) \sim b$ otherwise. 
\end{Def}

Recall that in the geometric realization of presimplicial sets the extra vertex $b$ is not added. 

\begin{Ex}\label{exinicial}
Consider the partial presimplicial set $\mathcal X = (X_n,d_i)$ given by the following sets and face maps:
$$X_0=\{\emptyset\}, \quad \quad X_1=\{r_0, r_1, r_2\}, \quad \quad X_2=\{T_0, T_1, T_2\}, $$ 
$$d_0(T_0)=r_2, \quad  d_2(T_0)=r_0, \quad d_0(T_1)=r_2, \quad d_1(T_1)=r_1,\quad  d_1(T_2)=r_1, \quad  d_2(T_2)=r_0.$$

\begin{figure}
\centering
\includegraphics[width = 11cm]{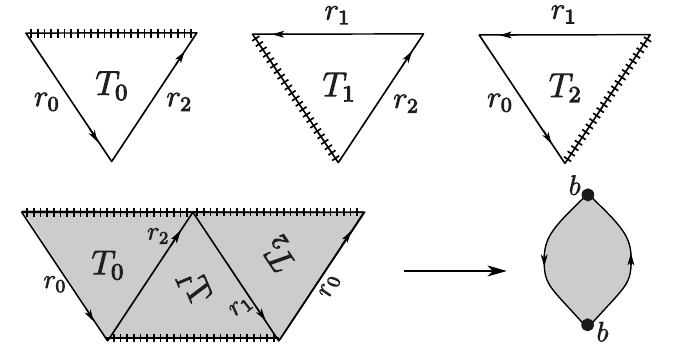}
\caption{\small{The geometric realization of the partial presimplicial set in Example \ref{exinicial}. Simplexes that must be collapsed to point $b$ are represented by dotted lines. After collapsing them one gets $\mathbb{R}P^2$.}}
\label{projective}
\end{figure}

The geometric realization of $\mathcal{X}$, $|\mathcal{X}|$, which is illustrated in Figure \ref{projective}, is homeomorphic to the projective plane $\mathbb{R}P^2$. 
\end{Ex}

\section{Khovanov homology and almost-extreme complex}\label{seckhovanov}

In this section we review the main ideas of Khovanov homology of unoriented (framed) links following Viro \cite{Vir1, Vir2}, and introduce the notion of almost-extreme Khovanov complex.

Let $D$ be an unoriented link diagram. A Kauffman state $s$ assigns a label, $A$ or $B$, to each crossing of $D$, i.e., $s: cr(D) \rightarrow \{A,B\}$. Let $\mathcal{S}$ be the collection of $2^c$ states of $D$, where $c=|cr(D)|$. For $s \in \mathcal{S}$, write $\sigma(s) = |s^{-1}(A)| - |s^{-1}(B)|$. We denote by $sD$ the system of circles and chords obtained after smoothing each crossing of $s$ according to its label by following the convention in Figure \ref{markers}. Write $|sD|$ for the number of circles in $sD$. 

\begin{figure}
\centering
\includegraphics[width = 12cm]{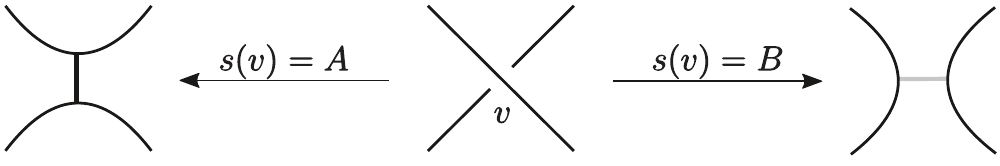}
\caption{\small{Smoothing of a crossing according to its $A$ or $B$-label. $A$-chords (resp. $B$-chords) are represented by dark (resp. light) segments.}}
\label{markers}
\end{figure}

Enhance a state $s$ with a map $\epsilon$ which associates a sign $\epsilon_i = \pm 1$ to each of the $|sD|$ circles in $sD$. We keep the letter $s$ for the enhanced state $(s,\epsilon)$ to avoid cumbersome notation. Write $\tau(s) = \sum_{i=1}^{|sD|} \epsilon_i$, and define, for the enhanced state $s$, the integers
$$
i(s) = \sigma(s), \quad j(s) = \sigma(s) + 2\tau(s).
$$

Given $s$ and $t$ two enhanced states of $D$, $t$ is said to be adjacent to $s$ if $i(t) = i(s) - 2$, $j(t) = j(s)$, both of them have identical labels except for one crossing $x$, where $s$ assigns an $A$-label and $t$ a $B$-label, and they assign the same sign to the common circles in $sD$ and $tD$.

Note that the circles which are not common to $sD$ and $tD$ are those touching the crossing $x$. In fact, passing from $sD$ to $tD$ can be realized by either merging two circles into one, or splitting one circle into two. The adjacency condition is equivalent to saying that $t$ is adjacent to $s$ if and only if they differ in one of the six ways shown in Figure \ref{differentials}.

\begin{figure}
\centering
\includegraphics[width = 14cm]{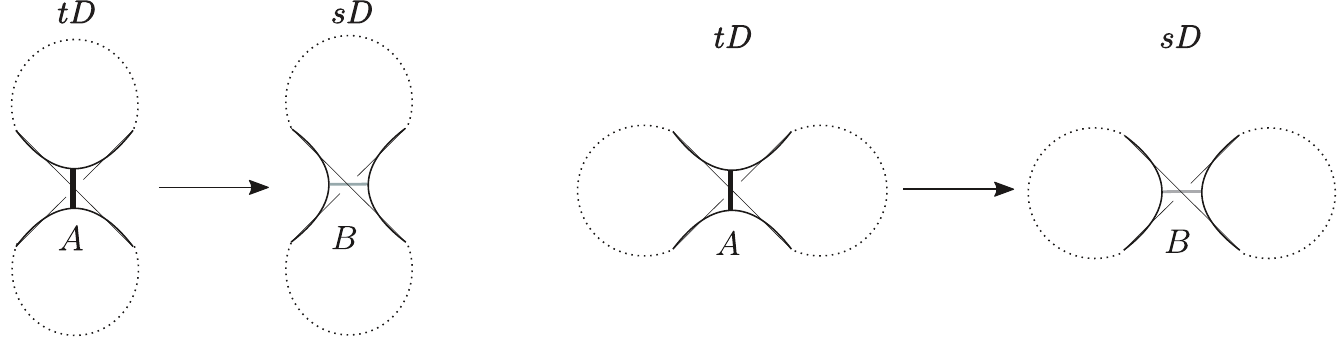}
\caption{\small{All possible enhancements when merging two circles are: $(+,- \rightarrow +)$, $(-,+ \rightarrow +)$, $(-,- \rightarrow -)$. The possibilities for a splitting are: $(+ \rightarrow +,+)$, $(- \rightarrow +,-)$ or $(- \rightarrow -,+)$.}}
\label{differentials}
\end{figure}

Let $C_{i,j}(D)$ be the free abelian group generated by the set of enhanced states $s$ of $D$ with $i(s) = i$ and $j(s) = j$. Order the crossings of $D$. For each integer $j$ consider the degree $-2$ complex
$$
\cdots \quad \longrightarrow \quad C_{i,j}(D) \quad \stackrel{\partial_i}{\longrightarrow} \quad C_{i-2,j}(D) \quad \longrightarrow \quad \cdots
$$
with differential $\partial_i(s) = \sum\limits_{t \in \mathcal{S}} (s:t) t$, with $(s:t) = 0$ if $t$ is not adjacent to $s$, and otherwise $(s:t) = (-1)^k$, with $k$ the number of $B$-labeled crossings in $s$ coming after crossing~$x$ in the chosen ordering. 

With the notation above, $\partial_{i-2} \circ \partial_i = 0$ and the corresponding homology groups
$$
H_{i,j}(D)=\frac{\textnormal{ker} (\partial_{i,j})}{\textnormal{im}(\partial_{i+2,j})}
$$
are the (framed) Khovanov homology groups of $D$, and they categorify the Kauffman bracket polynomial.

These groups are invariant under Reidemeister II and III moves. The effect of Reidemeister I move in Khovanov homology is a shift in the gradings, namely 
$$H_{i+1,j+3} \left(\raisebox{-5pt}{
\begingroup%
  \makeatletter%
  \providecommand\color[2][]{%
    \errmessage{(Inkscape) Color is used for the text in Inkscape, but the package 'color.sty' is not loaded}%
    \renewcommand\color[2][]{}%
  }%
  \providecommand\transparent[1]{%
    \errmessage{(Inkscape) Transparency is used (non-zero) for the text in Inkscape, but the package 'transparent.sty' is not loaded}%
    \renewcommand\transparent[1]{}%
  }%
  \providecommand\rotatebox[2]{#2}%
  \ifx\svgwidth\undefined%
    \setlength{\unitlength}{18.70443773bp}%
    \ifx\svgscale\undefined%
      \relax%
    \else%
      \setlength{\unitlength}{\unitlength * \real{\svgscale}}%
    \fi%
  \else%
    \setlength{\unitlength}{\svgwidth}%
  \fi%
  \global\let\svgwidth\undefined%
  \global\let\svgscale\undefined%
  \makeatother%
  \begin{picture}(1,0.91836368)%
    \put(0,0){\includegraphics[width=\unitlength,page=1]{R1-minus.pdf}}%
  \end{picture}%
\endgroup%
}\right) = H_{i,j}\left(\raisebox{-5pt}{
\begingroup%
  \makeatletter%
  \providecommand\color[2][]{%
    \errmessage{(Inkscape) Color is used for the text in Inkscape, but the package 'color.sty' is not loaded}%
    \renewcommand\color[2][]{}%
  }%
  \providecommand\transparent[1]{%
    \errmessage{(Inkscape) Transparency is used (non-zero) for the text in Inkscape, but the package 'transparent.sty' is not loaded}%
    \renewcommand\transparent[1]{}%
  }%
  \providecommand\rotatebox[2]{#2}%
  \ifx\svgwidth\undefined%
    \setlength{\unitlength}{18.70443773bp}%
    \ifx\svgscale\undefined%
      \relax%
    \else%
      \setlength{\unitlength}{\unitlength * \real{\svgscale}}%
    \fi%
  \else%
    \setlength{\unitlength}{\svgwidth}%
  \fi%
  \global\let\svgwidth\undefined%
  \global\let\svgscale\undefined%
  \makeatother%
  \begin{picture}(1,0.91836368)%
    \put(0,0){\includegraphics[width=\unitlength,page=1]{R1-zero.pdf}}%
  \end{picture}%
\endgroup%
}\right) = H_{i-1,j-3}\left(\raisebox{-5pt}{
\begingroup%
  \makeatletter%
  \providecommand\color[2][]{%
    \errmessage{(Inkscape) Color is used for the text in Inkscape, but the package 'color.sty' is not loaded}%
    \renewcommand\color[2][]{}%
  }%
  \providecommand\transparent[1]{%
    \errmessage{(Inkscape) Transparency is used (non-zero) for the text in Inkscape, but the package 'transparent.sty' is not loaded}%
    \renewcommand\transparent[1]{}%
  }%
  \providecommand\rotatebox[2]{#2}%
  \ifx\svgwidth\undefined%
    \setlength{\unitlength}{18.70443773bp}%
    \ifx\svgscale\undefined%
      \relax%
    \else%
      \setlength{\unitlength}{\unitlength * \real{\svgscale}}%
    \fi%
  \else%
    \setlength{\unitlength}{\svgwidth}%
  \fi%
  \global\let\svgwidth\undefined%
  \global\let\svgscale\undefined%
  \makeatother%
  \begin{picture}(1,0.91836368)%
    \put(0,0){\includegraphics[width=\unitlength,page=1]{R1-plus.pdf}}%
  \end{picture}%
\endgroup%
}\right).$$

\begin{Rem}
In this paper we work with the framed unoriented version of Khovanov homology. This version is equivalent to its oriented version, which categorifies the unreduced Jones polynomial. Framed and oriented versions are related by $$H_{i,j}(D) = H^{I, J}(\vec{D}),$$ 
where $I=\frac{w-i}{2}$ and $J= \frac{3w-j}{2}$, with $w = w(\vec{D})$ the writhe of the oriented diagram $\vec{D}$.
\end{Rem}

\begin{Def}
Given an unoriented link diagram $D$, let \vspace{-0.09cm} $$\vspace{-0.02cm} j_{\max}(D) = \max \{j(s) \, / \, s \mbox{ is an  enhaced state of } D\} \quad \mbox{and} \quad j_{\rm{almax}}(D) = j_{\max}(D) - 4.$$  We will refer to the complex $\{C_{i,j_{\max}}(D), \partial_i\}$ as the extreme Khovanov complex and to the corresponding homology groups $H_{i, j_{\max}}(D)$ as the extreme Khovanov homology groups of $D$. Analogously, we will call $\{C_{i,j_{\rm{almax}}}(D), \partial_i\}$ the almost-extreme Khovanov complex, and the groups $H_{i, j_{\rm{almax}}}(D)$ the almost-extreme Khovanov homology groups of $D$ (the values of $j$-index ``jump'' by four \footnote{In the oriented version of Khovanov homology the values of $J$ ``jump'' by two, and therefore $J_{\rm{almax}}(\vec{D}) = J_{\max}(\vec{D})-2$.}, and therefore the name of ``almost-extreme'' is justified).
\end{Def}

In \cite{GMS} the authors introduced a geometric realization of the complex $\{C_{i,j_{\max}}, \partial_i\}$ in terms of a simplicial complex constructed from the diagram $D$. In Section \ref{FromDtoXD}, starting from a semiadequate diagram $D$, we construct a partial presimplicial set $\mathcal{X}_D$ such that its cellular chain complex is equivalent to the almost-extreme Khovanov complex of $D$; as a consequence, almost-extreme Khovanov homology of $D$ can be computed as the homology of the geometric realization of $\mathcal{X}_D$.

\section{Semiadequate links}
Semiadequate links were introduced in \cite{LT} as a generalization of alternating links in the context of the proof of the first Tait conjecture.

\begin{Def}\cite{AsaedaPrzytycki}
Let $s$ be a state of a link diagram $D$. The state graph of $s$, $G_s(D)$, consists of vertices and edges in bijection with circles and chords in $sD$, respectively, in such a way that an edge between two vertices corresponds to a chord connecting the associated circles in $sD$. See Figure \ref{stategraph}. 
\end{Def} 

\begin{figure}
\centering
\includegraphics[width = 11.8cm]{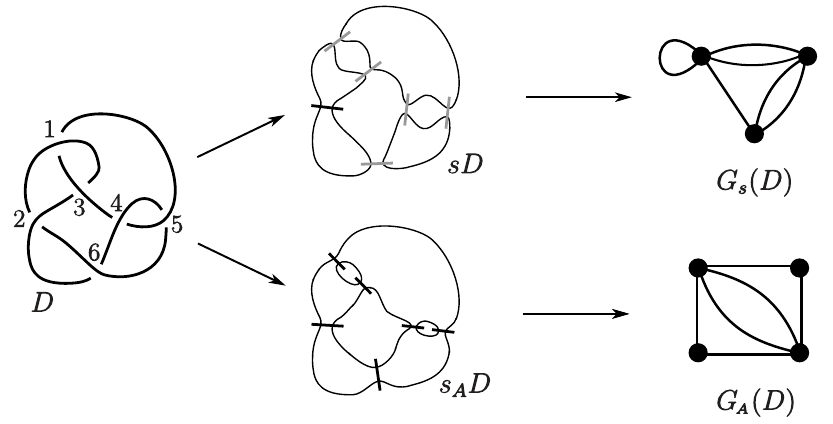}
\caption{\small{A diagram $D$ representing link $6^2_3$ in \cite{Rolfsen}, the resolutions corresponding to the states $s$ (mapping $1 \mapsto B, 2 \mapsto A, 3 \mapsto B, 4 \mapsto B, 5 \mapsto B, 6 \mapsto B)$ and $s_A$ and the corresponding state graphs $G_s(D)$ and $G_{s_A}(D)=G_A(D)$.}}
\label{stategraph}
\end{figure}

Notice that state graphs are planar. 


Given a link diagram $D$, let $s_A$ (resp. $s_B$) be the state assigning an $A$-label (resp. $B$-label) to every crossing. We write $G_A(D) = G_{s_A}(D)$ and $G_B(D) = G_{s_B}(D)$.

\begin{Def}\label{defAadequate}
A link diagram $D$ is $A$-adequate (resp. $B$-adequate) if $G_A(D)$ (resp. $G_B(D)$) contains no loops. $D$ is said to be adequate if it is both $A$-adequate and $B$-adequate. If $D$ is either $A$-adequate or $B$-adequate, it is called semiadequate. A link is said to be (semi)adequate if it admits a (semi)adequate diagram.  
\end{Def}

Without loss of generality, we consider semiadequate links to be $A$-adequate. Moreover, since Khovanov homology of the disjoint sum of two diagrams is given by K\"unneth formula from their individual Khovanov homology \cite{Kho}, from now on we assume that diagrams are connected.

\subsection{Khovanov homology of semiadequate links}\label{subkhovadequa}

In this section we present some facts about Khovanov homology of semiadequate links which will be useful for constructing the partial presimplicial set introduced in the next section.

In \cite{GMS} the states with maximal quantum index ($j_{\max}$) for a given diagram were characterized\footnote{We have adjusted the statement so it agrees with the framed version of Khovanov homology in Section~\ref{seckhovanov}.}: 

\begin{Prop}\rm{\cite{GMS}}\label{propmin}
Let $D$ be a link diagram with $c$ crossings and $s_A^+$ the state assigning an $A$-label to each of its crossing and a positive sign to every circle in $s_AD$. Then, $j_{\max}(D) = j(s_A^+)=c + 2|s_AD|$ and $j(s) = j_{\max}(D)$ if and only if $s \in S_{\max}$, where
$$S_{\max} = \{\mbox{enhanced states } s \, \big| \, |sD| = |s_AD| + |s^{-1}(B)|, \, \,  \tau(s) = |sD|\}.$$ 
\end{Prop}

The key point of Proposition \ref{propmin} is the fact that differentials taking part in the extreme complex $\{C_{i,j_{\max}}(D), \partial_i\}$ split a positive circle into two positive ones. This allows the authors to construct a simplicial complex whose cohomology equals the extreme Khovanov homology of~$D$. 

In the case of $A$-adequate diagrams the extreme Khovanov homology contains just one non-trivial group $H_{c,c+2|s_A|}(D)=\mathbb{Z}$ (see e.g. \cite{Kho, PS}). 

We are now interested in studying the almost-extreme Khovanov complex of $A$-adequate diagrams, which is of the form $$\begin{array}{cccccccc} 0 & \longrightarrow & C_{c, j_{\rm{almax}}}(D) &\stackrel{\partial_c}{\longrightarrow} & C_{c-2, j_{\rm{almax}}}(D) &\stackrel{\partial_{c-2}}{\longrightarrow} & \ldots \end{array}.$$

We give the following characterization of those states generating the almost-extreme Khovanov complex:

\begin{Prop}\label{charalmostextreme}
Let $D$ be an $A$-adequate link diagram. Then, $j(s) = j_{\rm{almax}}(D)$ if and only if $s \in S_{\rm{almax}}$, where
$$S_{\operatorname{almax}} = \left\{ \mbox{enhanced states } s \, \big| \, \begin{array}{lccc}
|sD| = |s_AD|  & \mbox{and } & \tau(s) = |sD|-2 & \mbox{if } s = s_A\\
|sD| =|s_AD|+|s^{-1}(B)|-2 & \mbox{and } & \tau(s) = |sD| & \mbox{if } s \neq s_A
\end{array}\right\}.$$ 
\end{Prop}

\begin{proof}
Recall that $j(s) = \sigma(s) + 2\tau(s)$, with $\sigma(s) = |s^{-1}(A)| - |s^{-1}(B)|$ and $\tau(s)= \sum_{i=1}^{|sD|} \epsilon_i$. Moreover, $j_{\rm{almax}}(D) = c + 2|s_AD|-4$.

Therefore, the states $s$ with $|s^{-1}(B)| = 0$ (that is, $s=s_A$) and all but one positive circles in $sD$ generate $C_{c,j_{\rm{almax}}}$. 

Next, notice that the $A$-adequacy of $D$ implies that when passing from $|s^{-1}(B)| = 0$ to $|s^{-1}(B)| = 1$ two circles merge into one, and therefore the generators of $C_{c-2,j_{\rm{almax}}}$ are those states $s$ with $|s^{-1}(B)|=1$ and all its $|sD| = |s_AD|-1$ circles labeled with a positive sign. Starting at this point, the differentials in the complex produce an splitting when increasing the number of $B$-labels, as occurs in the extreme case studied in \cite{GMS}. Applying Proposition \ref{propmin} to the characterization of the cases when $|s^{-1}(B)| \geq 2$ completes the proof.
\end{proof}

Proposition \ref{charalmostextreme} can be analogously stated for the case of $B$-adequate diagrams in terms of $s_B$, $|s^{-1}(A)|$ and $j_{\rm{almin}}(D) = j_{\min}(D)+4$.

Notice that Proposition \ref{charalmostextreme} implies that $C_{c, j_{\rm{almax}}}(D) = \mathbb{Z}^{|s_AD|}$ and $C_{c-2, j_{\rm{almax}}}(D) = \mathbb{Z}^{c}$. Now, since $D$ is $A$-adequate, the differential $\partial_c: \mathbb{Z}^{|s_AD|} \rightarrow \mathbb{Z}^{c}$ merges two circles with different signs into a positive one. Let $s \in C_{c, j_{\rm{almax}}}(D)$ be a state whose (only) negative circle is $\mathfrak{c}$. Then $\partial_c$ maps $s$ to  the sum taken over all states $t$ in $C_{c-2, j_{\rm{almax}}}(D)$ such that the only $B$-chord in $tD$ comes from an $A$-chord adjacent to $\mathfrak{c}$ in $s_AD$. 

A chord in a state $s$ is \emph{admissible} if it has both endpoints in the same circle of $sD$. For any generating state $t \in C_{c-2, j_{\rm{almax}}}(D)$, the $B$-chord in $tD$ is admissible, and the admissible $A$-chords are those having their endpoints in the same circle as the $B$-chord. 

To understand the behaviour of $\partial_{c-2}$ in the almost-extreme Khovanov complex, it suffices to think of the Khovanov subcomplex of each state $t \in  C_{c-2, j_{\rm{almax}}}(D)$ as the extreme complex of a link $D'$ whose associated $s_A^+$ state equals $t$, and apply the construction in \cite{GMS} based on admissible chords in previous paragraph.

The following example clarifies the previous reasoning:

\begin{Ex}\label{exilluminating}
Let $D$ be the $A$-adequate link diagram depicted in Figure \ref{figbig} and $G = G_A(D)$. Since $c = 5$ and $|s_AD| = 3$, $j_{\rm{almax}} = 5+2 \cdot 3 - 4 = 7$. 

\begin{figure}[t]
\centering
\includegraphics[width = 14.5cm]{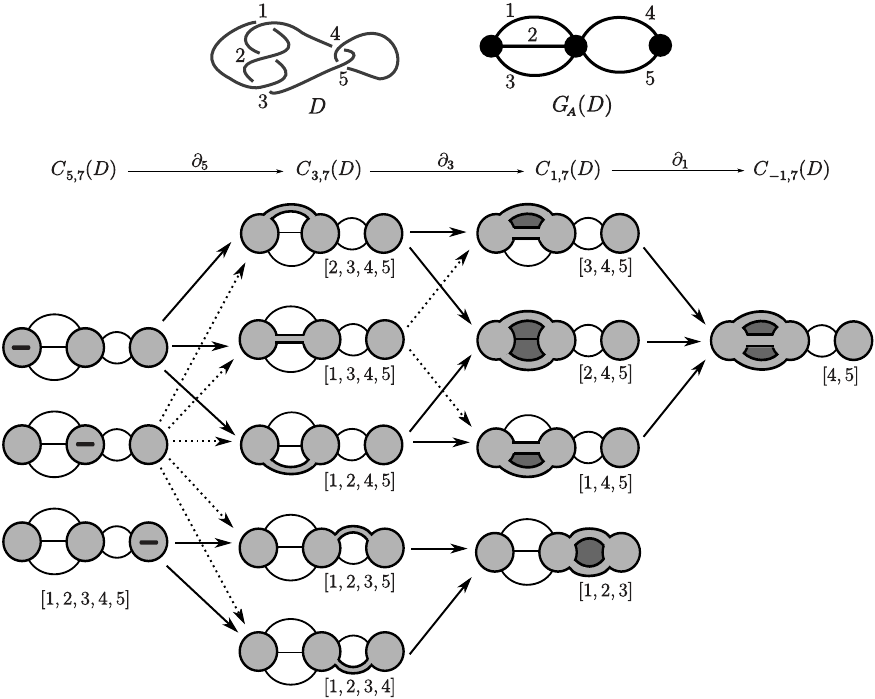}
\caption{\small{The $A$-adequate diagram $D$ illustrating Example \ref{exilluminating}, the state graph associated to $s_AD$, and the corresponding almost-extreme Khovanov complex. A non-labeled circle contains a positive sign; negative circles are labeled with $-$. For simplicity, $\text{$B$-chords}$ are not drawn.}}
\label{figbig}
\end{figure}

The three states generating $C_{5,7}(D)$ are obtained from $s_AD$ by associating a positive sign to every circle in $s_AD$ but one. A non-labeled circle contains a positive sign; negative circles are labeled with a minus ($-$) sign. 

The differential $\partial_5$ merges two circles with different signs into a positive one. Starting in $C_{3,7}$, all the circles are labeled with a positive sign, and therefore the complexes act as in the extreme case studied in \cite{GMS}, where just splittings are possible. 

If we label the chords by following the order of the crossings of $D$, then the bracket associated to each state in Figure \ref{figbig} correponds to its set of $A$-labels. This notation will be specially helpful in the reasoning in the proof of Theorem \ref{teorelating}.
\end{Ex}

\section{Partial presimplicial set from a semiadequate diagram}\label{FromDtoXD}

In this section, starting from a given semiadequate diagram, we describe an explicit combinatorial method for constructing a partial presimplicial set, and study the homotopy type of its geometric realization. In Section \ref{secproof} we will show that, in fact, the cellular chain complex of this geometric realization coincides with the almost-extreme Khovanov complex of the link diagram.

The key point of our construction consists of defining a partial presimplicial set associated to a loopless graph $G$. Even if in our settings we assume $G= G_A(D)$ for a given $A$-adequate diagram $D$, we stress that the construction below can be done for any loopless graph, not necessarily planar. In particular, Theorem \ref{teobipartito} holds also for loopless graphs, with $c$ denoting the number of edges. This may be useful in the context of categorification of links in thickened surface. 

Let $D$ be an $A$-dequate link diagram with $c=n+1$ (ordered) crossings and let $G = G_A(D)$, with $V(G) = \{T_0, \ldots, T_m\}$ and $E(G) = \{v_0 < \ldots < v_n\}$ the sets of vertices and edges of $G$, respectively. Recall that there is a one-to-one correspondence between the elements in $V(G)$ (resp. $E(G)$) and the the circles of $s_AD$ (resp. the crossings of $D$). 

We define a partial presimplicial set associated to $D$, $\mathcal{X}_D= (X_k, d_i)$, as follows:

Set $X_n = \{T_0, \ldots, T_m\}$. For each $0 \leq k < n$, define $X_k$ as the set of $(k+1)$-tuples of the form $(w_0, \ldots, w_k)$, with each $w_i < w_{i+1}$ corresponding to one of the edges of $G$ and satisfying that all the $n-(k+1)$ edges in $E(G) \setminus \{w_0, \ldots, w_k\}$ connect the same pair of vertices in the graph. In other words, the elements in $X_k$ consists of the ordered sequences of $k+1$ edges of $G$ such that the $n-(k+1)$ remaining edges are parallel. See Examples \ref{extrefoiltriangle} and \ref{ex8}.

The face maps are defined as
\small{$$d_i(T_j)=d_{i,n}(T_j)= \\
\left\{ 
\begin{array}{cl}
                               (v_0, \ldots, \hat v_i, \ldots, v_{n}) & \hbox{if the edge $v_i$ is adjacent to the vertex $T_j$ in $G$,} \\
                               0 & \hbox{otherwise.}
                             \end{array}
                           \right.$$ }
                           
For the case when $0<k<n$ the face maps are given by 
\small{$$d_i(w_0, \ldots, w_k)=d_{i,k}(w_0, \ldots, w_k)= $$
$$\left\{ 
\begin{array}{cl}
                               (w_0, \ldots,\hat w_i, \ldots, w_{k}) & \hbox{if the edge $w_i$ connects the same pair of vertices as $E(G) \setminus \{w_0, \ldots w_k\}$,} \\
                               0 &  \hbox{otherwise.}
                             \end{array}
                           \right.$$}

\begin{Ex}\label{extrefoiltriangle}
Consider the standard left-handed trefoil diagram, whose associated graph $G_A$ is a triangle. The associated partial presimplicial set and its geometric realization were described in Example \ref{exinicial}, with $r_0= v_0v_1$, $r_1=v_0v_2$ and $r_2= v_1v_2$, and therefore $|\mathcal{X}_{3_1}|=\mathbb{R}P^2$.
\end{Ex}

\begin{Ex}\label{ex8}

\begin{figure}[t]
\centering
\includegraphics[width = 12.5cm]{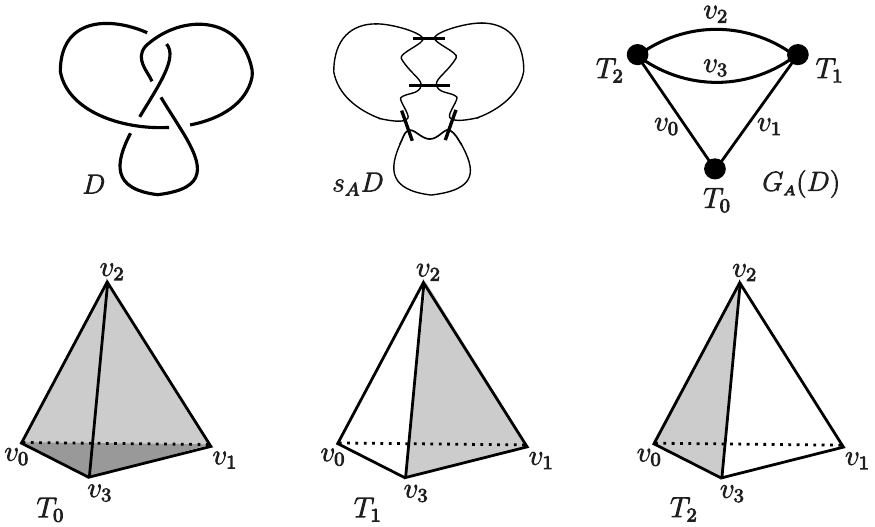}
\caption{\small{The diagram $D$ illustrating Example \ref{ex8}, its associated $s_AD$ and $G_A(D)$. The second row shows the three tetrahedra involved in the geometrization of $\mathcal{X}_D$. In the final step of the process simplexes with the same label must be identified and the shaded regions collapsed to point $b$.}}
\label{ejocho}
\end{figure}

Consider the diagram $D$ of the figure eight knot depicted in Figure \ref{ejocho}. Its associated resolution $s_AD$ contains 3 circles, $T_0, T_1, T_2$ together with $4$ chords corresponding to the $4$ crossings of $D$, and therefore $G_A(D)$ contains $3$ vertices and $4$ edges. The sets and the non-trivial face maps constituting the corresponding $\mathcal{X}_D = (X_n, d_i)$ are the following:
$$X_0 = \emptyset \quad \quad  X_1 = \{(v_0, v_1)\} $$
$$X_2 = \{(v_0, v_1, v_2), \, (v_0, v_1, v_3), \, (v_0, v_2, v_3), \, (v_1, v_2, v_3) \} \quad \quad X_3 = \{T_0, T_1, T_2\}$$

$$d_0(T_0) = (v_1,v_2,v_3) \quad d_1(T_0)= (v_0,v_2,v_3) \quad d_1(T_1) = (v_0,v_2,v_3) \quad d_2(T_1) = (v_0,v_1,v_3)$$  $$d_3(T_1) = (v_0, v_1,v_2) \quad  d_0(T_2) = (v_1,v_2,v_3) \quad d_2(T_2) = (v_0,v_1,v_3) \quad d_3(T_2) = (v_0,v_1,v_2)$$
$$d_2(v_0,v_1,v_2)= (v_0,v_1) \quad d_2(v_0,v_1,v_3) = (v_0,v_1)$$

To obtain the geometric realization of $\mathcal{X}_D$, we start with three tetrahedra labelled $T_0$, $T_1$ and $T_2$, and identify them by following the gluing instructions given by the face maps. An intermediate result before gluing and performing the contractions is illustrated in Figure \ref{ejocho}. Note that the shaded regions are collapsed to point $b$ in the next step, since they correspond to simplexes with trivial image under the face maps. After identifying simplexes with the same labels we get 
$$|\mathcal{X}_D| \sim \Sigma \mathbb{R}P^2.$$ 
\end{Ex}

\begin{Rem}\label{remrighttrefoil}
In the particular case when $G_A(D)$ consists on two vertices with $c>0$ edges connecting them (i.e., $m=1, n= c-1$), the resulting $\mathcal{X}_D$ is a presimplicial set, which implies that there is no need to collapse faces. In fact, $\mathcal{X}_D$ consists of two simplexes $\Delta^{c-1}$ whose boundaries are glued by the identity map, leading to $S^{c-1}$. The standard right-handed trefoil diagram, with $|\mathcal{X}_{\overline{3_1}}| \sim S^2$, is such an example. 
\end{Rem}

The next result shows that the graph $G_A(D)$ associated to the state $s_A$ determines the homotopy type of $|\mathcal{X}_D|$. 

\begin{Teo}\label{teobipartito}
Let $D$ be a connected $A$-adequate diagram and let $H$ be the simple graph obtained from $G=G_A(D)$ by replacing each multiedges by a single edge (that is, $v_i$ and $v_j$ are connected by one edge in $H$ iff they are connected by at least one edge in $G$). Let $\mathcal{X}_D$ be its associated partial presimplicial set defined as indicated above. If $D$ is a diagram with $c$ crossings and $p_1$ is the cyclomatic number of $H$, then 
$$|\mathcal{X}_D| \sim  \left\{ 
\begin{array}{lll}
\bigvee_{p_1}S^{c-2} \vee S^{c-1}& & \mbox{if } G \mbox{ is bipartite,} \\
\bigvee_{p_1 - 1} S^{c-2} \vee \sum^{c-3} \mathbb{R}P^2 & & \mbox{otherwise.}
\end{array}
\right.$$
\end{Teo}

The proof of Theorem \ref{teobipartito} uses the following well-known lemma, which can be proven by using general position argument. 

\begin{Lema}\label{lemmaproof} 
Let $D^n$, $n> 2$, be an oriented $n$-dimensional disk with $2k$ disks of dimension $n-1$ embedded in its boundary $\partial D^n$. We assume these disks to have disjoint interiors and to inherit the orientation of $D^n$. Let $K$ be a space obtained from $D^n$ by identifying these disks in pairs by homeomorphisms and contracting the rest of the boundary and the boundaries of the $2k$ $(n-1)$-dimensional disks to a point\footnote{Note that $K$ is homotopy equivalent to a $CW$-complex with a single vertex and cells in demension $n$ and $n-1$, thus a Moore space.}. Then 
\begin{enumerate}
\item If all identifications maps are orientation reversing, then the space $K$ is homotopy equivalent to the wedge of spheres $$\bigvee_k S^{n-1} \vee S^n.$$
\item If at least one identification map is orientation preserving, then the space $K$ is homotopy equivalent to $$\bigvee_{k-1} S^{n-1} \vee \sum^{n-2} \mathbb{R}P^2.$$
\end{enumerate}
\end{Lema}

\begin{Rem}
The case when $n=1$ implies that either $k=0$ or $1$, and in both cases $K=S^1$.\\ If $n=2$, then $K \sim F \vee \bigvee_{p-1} S^1$. More precisely, $K$ is homotopy equivalent to a $2k$-gon with sides identified in pairs to get surface $F$ and then all vertices left after this identification (say $p$ of them) are further identified to a point. There are two possible cases:
\begin{enumerate} 
\item[(1)] If all identifications are orientation reversing, then $F=F_g$ where $F_g$ is an oriented surface of genus $g=\frac{k-p+1}{2}$. Notice that the suspension $\Sigma F_g \sim S^3 \vee \bigvee_{2g}S^2$, compare to \cite[Chapter 4.1]{Hatcher}.
\item[(2)] If at least one identification is orientation reversing, then $F=F_{g'}= \#_{g'}RP^2$, where $g'= k-p+1$. Notice that the suspension $\Sigma F_{g'} \sim \Sigma RP^2 \vee \bigvee_{g'-1}S^2$.
\end{enumerate}
\end{Rem}

\begin{proof}[Proof of Theorem \ref{teobipartito}]
First, assume that $c>3$, so Lemma \ref{lemmaproof} can be applied. Recall that, by construction, $|\mathcal{X}_D|$ is built from $m+1$ copies of $\Delta^{c-1}$, and each $(c-2)$-dimensional (maximal) face of each $\Delta^{c-1}$ corresponds to an edge in $G$. After making all identifications and collapses induced by the face maps, each maximal face survives in two of the $m+1$ copies.

Assume now the case when $G=H$, that is, $G$ is a simple graph, which implies that $X_k$ is empty when $k<c-2$. Let $\mathcal{T}_G$ be a spanning tree for $G$ (i.e., a subgraph of $G$ containing all its vertices and no cycles). In order to construct $|\mathcal{X}_D|$, start by gluing the maximal faces of the $m+1$ copies of $\Delta^n$ according to the rules given by those face maps concerning the edges in $\mathcal{T}_G$. The result is homeomorphic to a ball $D^{c-1}$. 

Next, we need to perform the $p_1$ identifications which are left, which consists of gluing in pairs $2p_1$ $(c-2)$-dimensional disks embedded in $\partial D^{c-1}$, and collapse the rest of the boundary to a point. Lemma~\ref{lemmaproof} completes the proof.

In the case when $G$ contains multiedges the reasoning is analogous: Suppose there are $k$ edges connecting two vertices in $G$, and write $f_1, \ldots f_t$ for the $(c-2)$-dimensional faces of $\Delta^{c-1}$ associated to each of them\footnote{Reader familiar with Pachner moves will notice the similarity with Pachner simplex addition \cite{Pachner}.}. One can think on the $t$ dentifications corresponding to these faces as one single identification, by considering the $(c-2)$-dimensional face $f_1\cup \ldots \cup f_t$. This single identification preserves/reverses the orientation if and only if the $t$ individual identifications do. This completes the proof for the case when $c>3$.

If $c\leq 3$, then $G$ is either a triangle or $p_1=0$. The case when $G$ is a triangle corresponds to the case when $D$ is the standard left handed trefoil diagram, and therefore $|{\mathcal X}_D| = RP^2$, as shown in Example~\ref{extrefoiltriangle}. Finally, if $c\leq 3$ and $p_1=0$, then we check case by case that in the 7 possibilities we get ${\mathcal X}_D =S^{c-1}$. Notice that in the case of the trivial knot diagram with $c=0$, the graph $G=\bullet$ and $|{\mathcal X}_{\bigcirc}| =\emptyset = S^{-1}$.
\end{proof}

\begin{Cor}
Consider the diagram $D\#T(2,q)$, that is, the connected sum of a diagram $D$ and the standard diagram of the torus link of type $(2,q)$, with $q>0$ (see Figure \ref{figbig} for $q=2$). Then, $$|{\mathcal X}_{D\#T(2,q)}| \sim \Sigma^q|{\mathcal X}_{D}|.$$
In particular, the case $q=1$ corresponds to the positive Reidemeister I move.  
\end{Cor}

\section{Almost-extreme Khovanov homology via partial presimplicial sets}\label{secproof}

In this section we relate the partial presimplicial set constructed in Section \ref{FromDtoXD} with the almost-extreme Khovanov complex of the associated diagram.

\begin{Teo}\label{teorelating}
Let $D$ be an $A$-adequate diagram 
and let $\mathcal{X}_D$ be the partial presimplicial set constructed in Section \ref{FromDtoXD}. 
Then, the almost-extreme Khovanov complex of $D$ is chain homotopy equivalent to the cellular chain complex of $|\mathcal{X}_D|$.
\end{Teo}

\begin{proof}
The proof relies on the fact that we defined the sets and face maps of $\mathcal{X}_D$ according to the behaviour of the almost-extreme Khovanov complex of semiadequate links explained in Subsection \ref{subkhovadequa}:

The $n$-dimensional simplex $T_l$ represents the state of $C_{c,j_{\rm{almax}}}(D)$ having a negative sign in the associated circle $\mathfrak{c}_l$, for $0\leq l \leq m$.  

Now recall that for each $k$, $X_k$ consists of the ordered sequences of $k+1$ edges of $G=G_A(D)$ such that the $n-k$ remaining edges are parallel. See Example \ref{exilluminating}. Since there is a one-to-one correspondence between the edges of $G$ and the crossings in $D$, the $k+1$ edges appearing in the sequence corresponds to the $A$-labelled crossings of $D$. The condition that the $n-k$ remaining edges must be parallel and the definition of the face maps $d_i$ come from the fact that, appart from the starting merging given by $\partial_{c}$, the differentials in the almost-extreme complex consists of the splitting of one positive circle into two positive ones. 
\end{proof}

Theorems \ref{teobipartito} and \ref{teorelating} imply Theorem \ref{teomain}. As a consequence, we recover the following result from \cite{PPS,PS} for connected diagrams:

\begin{Cor}\label{corkhova}
The almost-extreme Khovanov homology of $D$ can be computed as the homology groups of $|\mathcal{X}_D|$. Namely, the non-trivial almost-extreme Khovanov homology groups of $D$ are 
\begin{itemize}
\item[(i)] If $G$ is bipartite, $H_{c,c+2|s_A|-4}(D)= \mathbb{Z}$ and $H_{c-2,c+2|s_A|-4}(D)= \mathbb{Z}^{p_1}$.
\item[(ii)] If $G$ contains odd cycles, $H_{c-2,c+2|s_A|-4}(D)= \mathbb{Z}^{p_1-1} \oplus \mathbb{Z}_2$.
\end{itemize}
\end{Cor}

\begin{figure}
\centering
\includegraphics[width = 12.5cm]{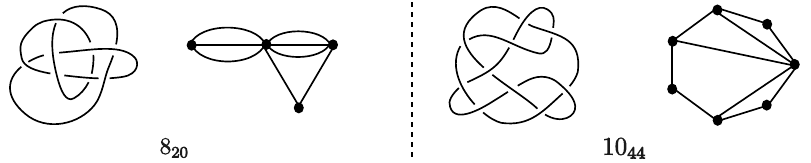}
\caption{\small{Diagrams of knots $8_{20}$ and $10_{44}$ from Rolfsen's table \cite{Rolfsen} and their associated $G_A$ graphs.}}
\label{grafos}
\end{figure}

\begin{Ex}
For a given link diagram $D$, the ranks of the (framed) Khovanov homology groups $H_{i,j}(D)$ can be arranged into a table with columns and rows indexed by $i$ and $j$, respectively. Tables \ref{Table3}-\ref{Table1} are taken from Knot Atlas \cite{Kno} with a modification so they match framed setting. An entry of the form $a$ represents the free group $\mathbb{Z}^a$, while $a_c$ represents the group $\mathbb{Z}^a_c$. We highlight the almost-extreme row. 

\begin{table}[t]
\centering
\includegraphics[width = 12.1cm]{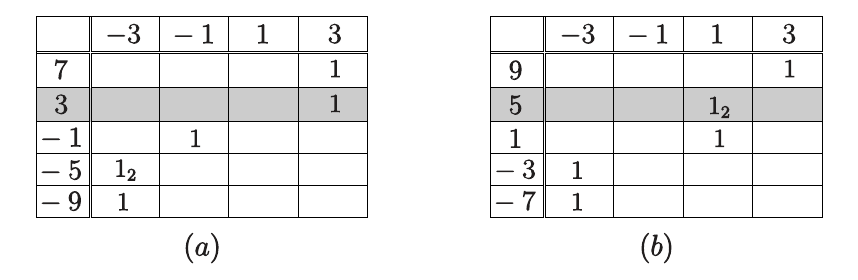}
\vspace{-0.2cm}
\caption{\small{Framed Khovanov homology of the standard diagrams of right-handed trefoil link $(a)$ and left-handed trefoil link $(b)$.}}
\label{Table3}
\end{table}

\begin{table}[t]
\centering
\includegraphics[width = 12.1cm]{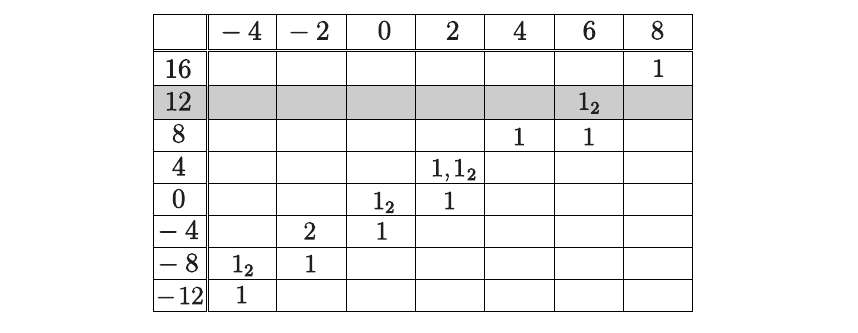}
\vspace{-0.2cm}
\caption{\small{Framed Khovanov homology of the diagram of knot $8_{20}$ in \cite{Rolfsen}.}}
\label{Table2}
\end{table}

\begin{table}[t]
\centering
\includegraphics[width = 12.1cm]{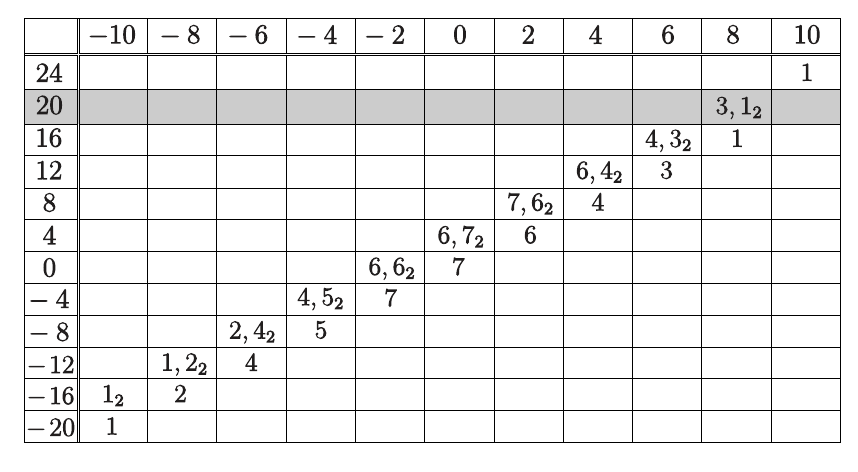}
\vspace{-0.2cm}
\caption{\small{Framed Khovanov homology of the diagram of knot $10_{44}$ in \cite{Rolfsen}.}}
\label{Table1}
\end{table}

In Table \ref{Table3} we present the Khovanov homology groups for the standard diagrams of right-handed trefoil $(a)$ and left-handed trefoil $(b)$. In Remark \ref{remrighttrefoil} we studied the case of the right-handed trefoil $\overline{3_1}$, whose associated $G_A(\overline{3_1})$ consists of two vertices with $3$ parallel edges, leading to $|\mathcal{X}_{\overline{3_1}}|= S^2$. Therefore, Corollary~\ref{corkhova} implies that the only non-trivial almost-extreme Khovanov homology group is $H_{3,3}(\overline{3_1})= \mathbb{Z}$. In the case of left-handed trefoil $3_1$, with $G_A(3_1)$ a triangle as shown in Example \ref{extrefoiltriangle}, Corollary \ref{corkhova} leads to $H_{1,5}(3_1)=\mathbb{Z}_2$. These results agree with Table \ref{Table3}.

Consider now the standard diagram of non-alternating knot $8_{20}$ from Rolfsen's table, which is $A$-adequate, and whose associated graph $G_A(8_{20})$ has $4$ vertices and a cycle of length $3$, as shown in Figure~\ref{grafos}$(a)$. Corollary \ref{corkhova} implies that the only non-trivial Khovanov homology group in its almost-extreme complex is $H_{6,12}(8_{20})= \mathbb{Z}_2$, which matches Table \ref{Table2}.

Finally, consider the standard diagram of knot $10_{44}$, whose associated graph $G_A(10_{44})$ is loopless and has $7$ vertices, as shown in Figure \ref{grafos}$(b)$. As $p_1=4$, $H_{8,20}(10_{44})= \mathbb{Z}^3 \oplus \mathbb{Z}_2$ and all other almost-extreme Khovanov homology groups are trivial, which agrees with Table \ref{Table1}.
\end{Ex}

\begin{Rem}
When applying Corollary \ref{corkhova} to an oriented $A$-adequate link diagram $D$ with $n$ negative crossings, the non-trivial almost-extreme Khovanov homology groups are $H^{-n,c-3n-|s_A|+2}(D)= \mathbb{Z}$ and $H^{-n+1,c-3n-|s_A|+2}(D)= \mathbb{Z}^{p_1}$ in the case when $G_A(D)$ is bipartite, and $H^{-n+1,c-3n-|s_A|+2}(D) = \mathbb{Z}^{p_1-1} \oplus \mathbb{Z}_2$ otherwise.
\end{Rem}

\vspace{0.1cm}

\noindent \textbf{Acknowledgements:} J\'ozef H. Przytycki is partially supported by the Simons Foundation Collaboration Grant for Mathematicians - 316446 and CCAS Dean’s Research Chair award. Marithania Silvero is partially supported by MTM2016-76453-C2-1-P and FEDER, and acknowledges financial support from the Spanish Ministry of Economy and Competitiveness, through the Mar\'ia de Maeztu Programme for Units of Excellence in R\&D (MDM-2014-0445).  The authors are grateful to the Department of Mathematics of the University of Barcelona for its hospitality.

\vspace{0.8cm}


\end{document}